\documentclass[a4, 11pt]{article}

\usepackage{amsmath,amsthm,amssymb}
\usepackage{mathtools} %typeset := properly
\usepackage{url}
\usepackage{cite}
\usepackage{xcolor}
\usepackage{hyperref}

\providecommand{\keywords}[1]
{
	\small	
	\textbf{Keywords:} #1
}

\title{A counterexample to Fuglede's conjecture in $(\mathbb{Z}/p\mathbb{Z})^4$ \\ for all odd primes}
\author{Sam Mattheus\footnote{sam.mattheus@vub.ac.be. Department of Mathematics, Pleinlaan 2, 1050 Brussels, Belgium.}  \\
	\small Vrije Universiteit Brussel \\
}
\date{}

\newtheorem{theorem}{Theorem}[section]
\newtheorem{prop}[theorem]{Proposition}

\newtheorem{lemma}[theorem]{Lemma}

\newtheorem{conj}[theorem]{Conjecture}

\theoremstyle{definition}
\newtheorem{defin}[theorem]{Definition}

\def\Z{\mathbb Z}
\def\R{\mathbb R}
\def\Q{\mathbb Q}

\def\Fp{\Z/p\Z}
\def\C{\mathbb{C}}

\newcommand\ZpZ[1]{(\Z/p\Z)^{#1}}

\begin{document}

\maketitle

\begin{abstract}
	In this short note we construct a spectral, non-tiling set of size $2p$ in $\ZpZ{4}$, $p$ odd prime. This example complements a previous counterexample in [C. Aten et al., \textit{Tiling sets and spectral sets over finite fields}, arXiv:1509.01090], which existed only for $p \equiv 3 \pmod{4}$. On the contrary we show that the conjecture does hold in $(\Z/2\Z)^4$.
\end{abstract}


\section{Origin and history of the conjecture}

In this paper we fill a gap in the investigation on Fuglede's conjecture for elementary abelian $p$-groups of rank four. The story of this conjecture begins in 1974, when Bent Fuglede \cite{Fuglede} investigates the relation between domains, i.e. sets of positive and finite Lebesgue measure in $\R^d$, that tile $\R^d$ by translation and the existence of an orthogonal exponential basis for the space of complex-valued functions on these domains. For future reference, a set satisfying this last property is called a spectral set and its spectrum is the corresponding orthogonal basis. He was able to show that both properties are equivalent, at least when the set is the fundamental domain of a lattice. On one of the last pages of this paper, he conjectures that the equivalence should in fact be true for any domain in $\R^d, d \geq 1$. Over the years, several authors managed to establish connections between the two properties, see for example \cite{IP,KL,L2002,LRW}, or were able to answer the question affirmatively for special kinds of domains. Notable results include those of Laba \cite{Laba}, who characterised in 2001  the union of two intervals satisfying the conjecture, and of Iosevich, Katz and Tao \cite{IKT} who proved in 2003 that the conjecture holds when the set is a convex planar domain.

In the light of these results, it is safe to say that history took an unexpected turn when in 2004 Terry Tao \cite{Tao} gave a counterexample in $\R^5$. To do so, he constructed a non-tiling spectral set, showing that at least one of the implications is false in $\R^5$, and by extension also in $\R^d, d \geq 5$. Interestingly, his approach first established a set in $(\Z/3\Z)^5$ of size 6, admitting a spectrum. As 6 does not divide $3^5$, this means that it cannot be tiling. After lifting this example to $\R^5$, while keeping the relevant properties intact, his result followed. His approach ignited the interest in the conjecture, not just in $\R^d$ anymore, but in the elementary abelian groups $(\Z/p\Z)^d$, $p$ prime, and it is this setting that we contribute to. We will properly state Fuglede's conjecture and its current status in $\ZpZ{d}$ in the next section, but wrap up the story for $\R^d$ first. 

Since Tao's result, Kolountzakis and Matolcsi \cite{KM} have shown that the reverse implication is also false for all $d \geq 5$. This was consequently improved to $d=4$ by Farkas and R\'{e}v\'{e}sz \cite{FR} and to $d=3$ by Farkas, Matolcsi and M\'ora \cite{FMM}. On the other hand, affirmative answers in small dimension still appear, for instance for the union of three intervals in $\R$ \cite{BKKM} and convex polytopes in $\R^3$ \cite{GL}, further exhibiting the disparity between small and large dimensions.

\section{Fuglede's conjecture in $\ZpZ{d}$}

We will now define what it means for a set $E \subseteq \ZpZ{d}$ to be tiling or to have a spectrum, in order to state the conjecture clearly. For the remainder of this paper, let $p$ denote a prime.

\subsection{Tiling sets}

\begin{defin}
	The set $E$ is \textbf{tiling} if there exists $T \subseteq \ZpZ{d}$ such that the sets $\{E+t \,\,|\,\, t \in T\}$ partition $\ZpZ{d}$. The set $T$ is called the \textbf{tiling partner}.
\end{defin}

Clearly, as $|E||T|=|\ZpZ{d}|=p^d$, it follows that $|E|=p^k, 0 \leq k \leq d$. The full space and singletons are trivially tiling sets.

Tiling sets can in fact be defined in any finite abelian group. This is an independent topic of interest and appears under the name of group factorizations. For the elementary abelian groups, one of the early exponents in this domain is the book by R\'edei \cite{Redeibook}. One of the highlights of this work is his proof that for $d=2$ a non-trivial tiling set is a subgroup or coset thereof, or its tiling partner is. This result can be leveraged into a proof of one implication of the Fuglede conjecture for $d=2$, as is also described in \cite{KDSV}. Moreover, he conjectured \cite[Problem 5]{Redeibook} that in three dimensions a similar situation occurs -- that at least one of the factors does not span the whole group -- although this problem remains open to this day. We refer to Sz\H onyi's paper \cite{Szonyi} for an introduction to R\'edei's result and its connections to geometry over finite fields. On the other hand, for factorizations of finite abelian groups in general, we refer to the book of Szab\'o \cite{Szabobook}.

Before we turn to the subject of the other half of the conjecture, we need to make a small detour.

\subsection{The space $L^2(F)$ and characters of finite abelian groups}

In order to state the definition of a spectral set, we need two components: the the vector space $L^2(F)$ of complex-valued functions on a finite set $F$, and characters of a finite abelian group. For the remainder of this section, fix $F$ to be a finite set.

\begin{defin}
	The vector space of all functions $F \rightarrow \mathbb{C}$ is denoted as $L^2(F)$, when equipped with the inner product defined by	
	\begin{align}\label{innerproductdef}\langle f,g \rangle_{L^2(F)} = \frac{1}{|F|}\sum_{x \in F}f(x)\overline{g(x)}. \end{align}
\end{defin}

Since the characteristic functions $\{\delta_x \,\, | \, \, x \in F\}$ are an orthogonal basis for $L^2(F)$, it is clear that $\dim(L^2(F))=|F|$.

A word of caution is necessary: there are several different conventions regarding the scaling of the sum, the motivation for the choice $1/|F|$ will be clear after we discuss some character theory.

\begin{defin}
	A \textbf{character} $\chi$ of a finite abelian group $G$ is a homomorphism $\chi:G\rightarrow \C$.
\end{defin}

There exists a rich literature on characters of finite abelian groups, so we will collect a few necessary results, focused on the elementary abelian groups, which are contained in any standard work in this area. In particular, we refer to \cite{Conrad} for an excellent introduction into this topic.

We will write elements $\textbf{x}=(x_1,\dots,x_d) \in \ZpZ{d}$ in bold, in contrast with elements of $\Z/p\Z$. The characters of $\ZpZ{d}$ are defined by 
\[\chi_\textbf{a}(\textbf{x}) = e^{\frac{2\pi i}{p}\langle \textbf{a},\textbf{x} \rangle},\]
for all $\textbf{a},\textbf{x} \in \ZpZ{d}$, where $\langle \textbf{a},\textbf{x} \rangle = a_1x_1 + \dots + a_dx_d$. Observe that $\overline{\chi_\textbf{a}}=(\chi_\textbf{a})^{-1}=\chi_\textbf{-a}$. 

\begin{lemma}[The orthogonality relations of characters]
	Let $\emph{\textbf{a}},\emph{\textbf{b}} \in \ZpZ{d}$ and  $\chi_\emph{\textbf{a}},\chi_\emph{\textbf{b}}$ their associated characters, then 
	\[\sum_{\emph{\textbf{x}} \, \in \ZpZ{d}}\chi_\emph{\textbf{a}}(\emph{\textbf{x}})\overline{\chi_\emph{\textbf{b}}(\emph{\textbf{x}})} = \begin{cases}
		p^d & \text{if } \emph{\textbf{a}}=\emph{\textbf{b}} \\
		0 & \text{if } \emph{\textbf{a}}\neq\emph{\textbf{b}}
		\end{cases}\]
\end{lemma}

This implies that the characters $\{\chi_\textbf{a} \,\, | \,\, \textbf{a} \in \ZpZ{d}\}$ are an orthonormal basis for $L^2(\ZpZ{d})$, motivating the choice of scalar $1/|F|$ in the inner product \eqref{innerproductdef}. Furthermore, we have a group operation on the characters as for any $\textbf{a},\textbf{b},\textbf{x} \in \ZpZ{d}$,
\[\chi_\textbf{a}(\textbf{x})\chi_\textbf{b}(\textbf{x}) = e^{\frac{2\pi i}{p}\langle \textbf{a},\textbf{x} \rangle}e^{\frac{2\pi i}{p}\langle \textbf{b},\textbf{x} \rangle} = e^{\frac{2\pi i}{p}\langle \textbf{a}+\textbf{b},\textbf{x} \rangle} =  \chi_{\textbf{a}+\textbf{b}}(\textbf{x}).\]

The existence of an orthonormal basis of characters is the defining property of spectral sets as we will now see. 

\subsection{Spectral sets}

\begin{defin}
	The set $E$ is \textbf{spectral} if there exists $A \subseteq \ZpZ{d}$ such that $\{\chi_\textbf{a} \,\, | \,\, \textbf{a} \in A\}$ is an orthogonal basis of $L^2(E)$. The set $A$ is called the spectrum.
\end{defin}

As $\dim(L^2(E))=|E|$, it is clear that $|A|=|E|$. Again, we see that the full space and the singletons satisfy this property trivially: the first due to what we saw above, the second because $\dim(L^2(\{\textbf{x}\}))=1$ and hence every non-zero complex number is an orthogonal basis for this space. 

We are now in the position to state Fuglede's conjecture in $\ZpZ{d}$. 

\begin{conj}
	Let $E \subseteq \ZpZ{d}$, then $E$ is a tiling set if and only if $E$ is a spectral set.
\end{conj}

We briefly summarize the current state of affairs.

\begin{itemize}
	\item[\underline{$d=1$}] Recall that in general, if $E$ tiles then $|E|$ divides $p^d$. Similarly, we will show later that if $E$ is a spectral set in $\Z/p\Z$, then $|E|=1$ or $p$. Therefore, either one of the properties implies that $E$ is a singleton or the full space, which shows that the conjecture is trivially true.
	\item[\underline{$d=2$}] First proved by Iosevich, Mayeli and Pakianathan \cite{IMP} to be true. Slightly alternative proofs have appeared in \cite{Researchpaper} and \cite{KDSV}.
	\item[\underline{$d=3$}] Half of the conjecture is known: `tiling implies spectral' is settled in the affirmative \cite{Researchpaper} while the other implication remains open. In the same paper, a computation by hand for $p=2,3$ shows that the full conjecture is true in these cases.
	\item[\underline{$d=4$}] A counterexample to `spectral implies tiling' for $p \equiv 3 \pmod{4}$ has been constructed in \cite{Researchpaper}, using the machinery of log-Hadamard matrices. Inspired by this construction, we extend this in the next section to all odd primes $p$. Before doing so, we show the perhaps surprising result that for $p=2$ the conjecture is true. However, this exception seems to be due to the small number of points and is not an indication for larger dimensions. These results were found independently and with different methods by Ferguson and Sothanaphan \cite{FS}.
	\item[\underline{$d\geq5$}] For all odd primes $p$ there exist non-tiling spectral sets \cite{Researchpaper}. When $p=2$, a counterexample for the same direction has been constructed for $d=10$ \cite{FS}. The converse is still wide open as is the case $p=2$ and $5 \leq d \leq 10$.
\end{itemize}

One can see that also in the case of elementary abelian groups, there is a discrepancy between small and large dimensions, analogous to the situation in $\R^d$. We give a short proof that the conjecture holds for dimension four when $p=2$, relying on \cite[Theorem 1.1]{Researchpaper}, of which we copy the relevant parts for convenience.

\begin{theorem}\label{mainthmresearchpaper}
	Let $E \subseteq \ZpZ{d}$, $p$ prime, then
	\begin{enumerate}
		\item If $E$ is a spectral set then $|E|\in\{1,p^d\}$ or $|E|=kp$ for some $1 \leq k \leq p^{d-2}$.
		\item If $E$ is a spectral set in $(\Z/2\Z)^d$, then $|E|\in\{1,2\}$ or is a multiple of $4$.
		\item If $|E| \in \{p,p^{d-1}\}$ then $E$ tiles if and only if $E$ is spectral.
		\item A set $E$ tiles with a subspace tiling partner $W$ if and only if $E$ is spectral with spectrum $W^\perp$. 
	\end{enumerate}
\end{theorem}

\newpage

\begin{prop}
	Let $E \subseteq (\Z/2\Z)^4$, then $E$ is a tiling set if and only if $E$ is a spectral set.
\end{prop}
\begin{proof}
	In fact, we will show that any set satisfying the necessary size restrictions is simultaneously tiling and spectral. Both for tiling sets and spectral sets, it follows from the restrictions in the theorem above that for non-trivial $E$, we have $|E| \in \{2,4,8\}$. In all cases except for $|E|=4$, the equivalence follows from (3), so suppose for the remainder of the proof that $|E|=4$.
	
	If $E$ is the set of four points of a two-dimensional plane, then $E$ clearly tiles the space with a subspace tiling partner. By part (4) of the theorem above, this means that $E$ is also spectral.
	
	Therefore, suppose that $E$ is not contained in a plane. By Corollaries 3.2 and 4.3 from \cite{Researchpaper}, the property of being a tiling set or spectral set remains unchanged after a change of basis. Therefore, after a transformation of the coordinates, we can suppose that $E = \{(0,0,0,0),(1,0,0,0),(0,1,0,0),(0,0,1,0)\}$. This set tiles with tiling partner $\{(0,0,0,0),(0,0,0,1),(1,1,1,0),(1,1,1,1)\}$, which turns out to be a subspace. By part (4), this implies that $E$ is also spectral. 		
\end{proof}

\section{The counterexample}

We return to the situation when $p$ is an odd prime. In this section we will construct a spectral set $E\subseteq \ZpZ{4}$ of size $2p$. As $2p$ does not divide $p^4$, this implies that $E$ cannot be tiling.

Let $p$ be an odd prime and fix a non-square $n \in \Fp$. Define the sets $E$ and $A$ in $\ZpZ{4}$ by

\[E = \left\{(t^2,t,t,1) \,\, | \,\, t \in \Fp\right\} \cup \left\{(nt^2,nt,t,n) \,\, | \,\, t \in \Fp\right\}\]
and
\[\hspace{.6cm}A = \left\{(1,2i,0,i^2) \,\, | \,\, i \in \Fp\right\} \cup \left\{(0,0,-2ni,ni^2	) \,\, | \,\, i \in \Fp\right\}.\]

Remark that when $p \equiv 3 \pmod{4}$ we can choose $n = -1$, which corresponds, up to invertible linear transformation, to the counterexample constructed in \cite{Researchpaper}. 

First, we reformulate the orthogonality condition for any $\textbf{a},\textbf{b} \in A$ 

\begin{align*}
0 &= \langle \chi_\textbf{a}, \chi_{\textbf{b}} \rangle_{L^2(E)} = \sum_{\textbf{e} \in E}\chi_\textbf{a}(\textbf{e})\overline{\chi_{\textbf{b}}(\textbf{e})} = \sum_{\textbf{e} \in E}\chi_{(\textbf{a}-\textbf{b})}(\textbf{e}) 
\end{align*}

The second ingredient is the following well-known lemma on the vanishing of $p$-th roots of unity for which we give a short proof.

\begin{lemma}\label{vanishing roots}
	Let $\xi = e^{\frac{2\pi i}{p}}$, then 
	$\sum_{i = 0}^{p-1}c_i\xi^i = 0, c_i \in \Q \Leftrightarrow \text{all $c_i$ are equal}$.
\end{lemma}
\begin{proof}
	The $p$-th roots of unity are exactly $\{1, \xi,\dots,\xi^{p-1}\}$, the roots of $X^p-1 \in \Q[X]$. As this polynomial factors over $\Q[X]$ into the irreducible factors $X-1$ and $X^{p-1}+\dots+X+1$ by Eisenstein's criterion (after the transformation $X \rightarrow X+1$), this directly implies that the only possibility for a sum of $p$-th roots to be zero is if all coefficients are equal.
\end{proof}

Using this lemma we can obtain a geometrical interpretation of spectral sets. When $|E| > 1$, a spectral set $E$ with spectrum $A$ is equidistributed along the hyperplanes of any parallel class defined by the linear equations
\[\langle \textbf{a}-\textbf{b}, \textbf{e} \rangle = c, \,\,c \in \Fp,\]
where $\textbf{a}$ and $\textbf{b}$ are fixed and distinct elements of $A$.
This implies in general, that a spectral set is a singleton or has size divisible by $p$, which is part of Theorem \ref{mainthmresearchpaper} (1).

We will now show that $\{\chi_\textbf{a} \,\, | \,\, \textbf{a} \in A\}$ is indeed an orthogonal set of functions in $L^2(E)$. As also $|E|=|A|$, this then implies that $A$ is a spectrum for $E$. 

We claim that for any $c \in \Fp$, the number of solutions in $\textbf{e}$ to the equation $\langle \textbf{a}-\textbf{b},\textbf{e} \rangle = c$ always equals two. If this holds, then the equalities

\[\sum_{\textbf{e} \in E}\chi_{(\textbf{a}-\textbf{b})}(\textbf{e}) = \sum_{\textbf{e} \in E} e^{\frac{2\pi i}{p}\langle \textbf{a}-\textbf{b},\textbf{e} \rangle} = \sum_{m=0}^{p-1} 2e^{\frac{2\pi i}{p}m} = 0,\]

applying Lemma \ref{vanishing roots} in the last equality, show that $\{\chi_\textbf{a} \,\, | \,\, \textbf{a} \in A\}$ is indeed an orthogonal set in $L^2(E)$. As $A$ is defined as the union of two one-parameter sets, the vector $\textbf{a}-\textbf{b}$ is one of the following expressions:

\[\textbf{a}-\textbf{b} = \begin{cases}
	(0,2(i-j),0,i^2-j^2) & (i \neq j) \\[2pt]
	(0,0,-2n(i-j),n(i^2-j^2)) & (i \neq j) \\[2pt]
	\pm(1,2i,2nj,i^2-nj^2) 
\end{cases} \hspace{10pt} \text{ where } i,j \in \Fp.\]

For each of the cases, we will show that $\langle \textbf{a}-\textbf{b},\textbf{e} \rangle = c$ has two solutions in $\textbf{e}$.

\begin{enumerate}
	\item {$\textbf{a}-\textbf{b} = (0,2(i-j),0,i^2-j^2)$ where $i,j \in \Fp, i \neq j$}. \\
	The equation $\langle \textbf{a}-\textbf{b}, \textbf{e} \rangle = c$, can be rewritten as
	\begin{align*}
	2(i-j)t+(i^2-j^2)&=c \\
	2(i-j)nt+n(i^2-j^2)&=c. 
	\end{align*}
	As $i \neq j$, it is clear that each equation has exactly one solution in $t$ and so there are two solutions for $t$, and hence for $\textbf{e}$, in total.
	
	\item {$\textbf{a}-\textbf{b} = (0,0,-2n(i-j),n(i^2-j^2))$ where $i,j \in \Fp, i \neq j$} \\
	This case is similar to the previous one as the resulting equations are
	\begin{align*}
	-2n(i-j)t+n(i^2-j^2)&=c \\
	-2n(i-j)t+n^2(i^2-j^2)&=c. 
	\end{align*}
	
	\item {$\textbf{a}-\textbf{b} = \pm(1,2i,2nj,i^2-nj^2)$ where $i,j \in \Fp$} \\
	Now we find the union of two quadratic equations $Q_1(t)$ and $Q_2(t)$:
	\begin{align*}
	t^2+(2i+2nj)t+(i^2-nj^2)&=\pm c \\
	nt^2+(2ni+2nj)t+(ni^2-n^2j^2)&=\pm c. 
	\end{align*}
	Denoting the discriminants by $D_1$ and $D_2$ respectively, one can compute that 
	\[nD_1 = D_2 = 4n(2nij+(n^2+n)j^2\pm c) \]
	Recalling the fact that $n$ is a non-square in $\Fp$, this means that $Q_1(t)=0$ has two solutions if and only if $Q_2(t)=0$ has zero and vice versa, and if one has a unique solution, then the other does too. Collecting everything, this means we can always find two solutions again, which concludes the proof.  \hfill \qed
\end{enumerate}

\nocite{*}

\bibliographystyle{hplain.bst}	
\bibliography{bibliography}	

\end{document}